\documentclass[11pt, reqno]{amsart}
\usepackage[margin=1in]{geometry}
\usepackage{latexsym}
\usepackage{amsfonts}
\usepackage{amsmath}
\usepackage{amssymb}
\usepackage{amsthm}
\usepackage{enumerate}
\usepackage[hang,flushmargin]{footmisc}
\usepackage{caption}
\usepackage{tabu}
\usepackage{mathrsfs}
\usepackage{amsaddr}
\usepackage{subfig}

\newcommand{\Comp}{\operatorname{Comp}}
\newcommand{\Av}{\operatorname{Av}}
\DeclareMathOperator{\SW}{\mathsf{SW}}

\usepackage{graphicx}
\usepackage{epstopdf}
\usepackage{epsfig}
\usepackage{caption}
 
\usepackage{bm} 
\usepackage{cite}

\makeatletter
\newtheorem*{rep@theorem}{\rep@title}
\newcommand{\newreptheorem}[2]{%
\newenvironment{rep#1}[1]{%
 \def\rep@title{#2 \ref{##1}}%
 \begin{rep@theorem}}%
 {\end{rep@theorem}}}
\makeatother

\newtheorem{theorem}{Theorem}[section]
\newreptheorem{theorem}{Conjecture}
\newtheorem{lemma}{Lemma}[section]
\newtheorem{proposition}{Proposition}[section]
\newtheorem{corollary}{Corollary}[section]

\newtheorem{conjecture}{Conjecture}[section]

\newtheorem{problem}{Problem}[section]

\theoremstyle{definition}
\newtheorem{definition}{Definition}[section]
\newtheorem{remark}{Remark}[section]
\newtheorem{example}{Example}[section]

\DeclareMathOperator{\des}{des}

\DeclareMathOperator{\peak}{peak}
\DeclareMathOperator{\tl}{tl}

\DeclareMathOperator{\leg}{leg}

\begin{document}
\title{Counting 3-Stack-Sortable Permutations}
\author{Colin Defant}
\address{Princeton University \\ Fine Hall, 304 Washington Rd. \\ Princeton, NJ 08544}
\email{cdefant@princeton.edu}

\begin{abstract}
We prove a ``decomposition lemma" that allows us to count preimages of certain sets of permutations under West's stack-sorting map $s$. As a first application, we give a new proof of Zeilberger's formula for the number $W_2(n)$ of $2$-stack-sortable permutations in $S_n$. Our proof generalizes, allowing us to find an algebraic equation satisfied by the generating function that counts $2$-stack-sortable permutations according to length, number of descents, and number of peaks. This is also the first proof of this formula that generalizes to the setting of $3$-stack-sortable permutations. Indeed, the same method allows us to obtain a recurrence relation for $W_3(n)$, the number of $3$-stack-sortable permutations in $S_n$. Hence, we obtain the first polynomial-time algorithm for computing these numbers. We compute $W_3(n)$ for $n\leq 174$, vastly extending the $13$ terms of this sequence that were known before. We also prove the first nontrivial lower bound for $\lim\limits_{n\to\infty}W_3(n)^{1/n}$, showing that it is at least $8.659702$. Invoking a result of Kremer, we also prove that $\lim\limits_{n\to\infty}W_t(n)^{1/n}\geq(\sqrt{t}+1)^2$ for all $t\geq 1$, which we use to improve a result of Smith concerning a variant of the stack-sorting procedure. Our computations allow us to disprove a conjecture of B\'ona, although we do not yet know for sure which one. 

In fact, we can refine our methods to obtain a recurrence for $W_3(n,k,p)$, the number of $3$-stack-sortable permutations in $S_n$ with $k$ descents and $p$ peaks. This allows us to gain a large amount of evidence supporting a real-rootedness conjecture of B\'ona. Using part of the theory of valid hook configurations, we give a new proof of a $\gamma$-nonnegativity result of Br\"and\'en, which in turn implies an older result of B\'ona. We then answer a question of the current author by producing a set $A\subseteq S_{11}$ such that $\sum_{\sigma\in s^{-1}(A)}x^{\des(\sigma)}$ has nonreal roots. We interpret this as partial evidence against the same real-rootedness conjecture of B\'ona that we found evidence supporting. Examining the parities of the numbers $W_3(n)$, we obtain strong evidence against yet another conjecture of B\'ona. 
We end with some conjectures of our own.  
\end{abstract}

\maketitle

\bigskip

\section{Introduction}\label{Sec:Intro}

\subsection{The Stack-Sorting Map}
We use the word ``permutation" to refer to an ordering of a set of positive integers written in one-line notation. Let $S_n$ denote the set of permutations of the set $[n]:=\{1,\ldots,n\}$. If $\pi$ is a permutation of length $n$, then the \emph{normalization} of $\pi$ is the permutation in $S_n$ obtained by replacing the $i^\text{th}$-smallest entry in $\pi$ with $i$ for all $i\in[n]$. We say a permutation is \emph{normalized} if it is equal to its normalization. A \emph{descent} of a permutation $\pi=\pi_1\cdots\pi_n$ is an index $i\in[n-1]$ such that $\pi_i>\pi_{i+1}$. A \emph{peak} of $\pi$ is an index $i\in\{2,\ldots,n-1\}$ such that $\pi_{i-1}<\pi_i>\pi_{i+1}$. Let $\des(\pi)$ and $\peak(\pi)$ denote the number of descents of $\pi$ and the number of peaks of $\pi$, respectively. 

\begin{definition}\label{Def3}
Given $\tau\in S_m$, we say a permutation $\sigma=\sigma_1\cdots\sigma_n$ \emph{contains the pattern} $\tau$ if there exist indices $i_1<\cdots<i_m$ in $[n]$ such that the normalization of $\sigma_{i_1}\cdots\sigma_{i_m}$ is $\tau$. We say $\sigma$ \emph{avoids} $\tau$ if it does not contain $\tau$. Let $\Av(\tau^{(1)},\ldots,\tau^{(r)})$ denote the set of normalized permutations that avoid the patterns $\tau^{(1)},\ldots,\tau^{(r)}$. Let $\Av_n(\tau^{(1)},\ldots,\tau^{(r)})=\Av(\tau^{(1)},\ldots,\tau^{(r)})\cap S_n$. 
\end{definition}

The study of permutation patterns is now a major area of research; it began with Knuth's analysis of a certain ``stack-sorting algorithm" \cite{Knuth}. In his dissertation, West \cite{West} defined a deterministic variant of Knuth's algorithm. This variant is a function, which we call the ``stack-sorting map" and denote by $s$, that sends permutations to permutations. The stack-sorting map has now been studied extensively \cite{Albert, Bona, BonaSurvey, BonaWords, BonaSimplicial, BonaSymmetry, Bousquet98, Bousquet, Bouvel, BrandenActions, Branden3, Claessonn-4, Claesson, Cori, DefantCatalan, DefantDescents, DefantEnumeration, DefantFertility, DefantFertilityWilf, DefantPolyurethane, DefantPostorder, DefantPreimages, DefantClass, DefantEngenMiller, DefantKravitz, Dulucq, Dulucq2, Elder, Goulden, Hanna, Maya, Smith, Ulfarsson, West, Zeilberger}. The reader seeking further historical background and motivation should see one of the references \cite{Bona, BonaSurvey, DefantCatalan, DefantFertility, DefantPostorder, DefantPreimages, DefantClass, DefantEngenMiller}. 

To define the function $s$, let us begin with an input permutation $\pi=\pi_1\cdots\pi_n$. At any point in time during this procedure, if the next entry in the input permutation is smaller than the entry at the top of the stack or if the stack is empty, the next entry in the input permutation is placed at the top of the stack. Otherwise, the entry at the top of the stack is annexed to the end of the growing output permutation. This process terminates when the output permutation has length $n$, and $s(\pi)$ is defined to be this output permutation. The following illustration shows that $s(4162)=1426$. 

\begin{center}
\includegraphics[width=1\linewidth]{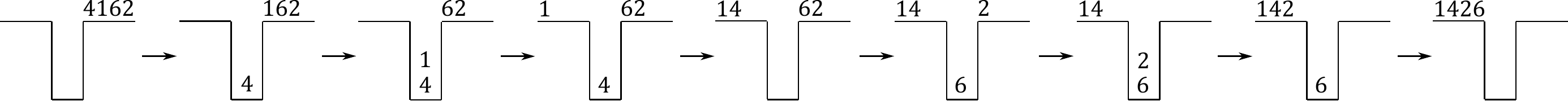}
\end{center}

\begin{definition}\!\!\!\footnote{The permutations that we call $t$-stack-sortable are often called ``West $t$-stack-sortable," but we have dropped the word ``West" for brevity. The term ``$t$-stack-sortable" is sometimes used to refer to permutations that can be sorted using Knuth's (nondeterministic) algorithm with $t$ stacks in series; that use of the term is different from ours.}\label{Def1}
We say a permutation $\pi$ is $t$-\emph{stack-sortable} if $s^t(\pi)$ is an increasing permutation, where $s^t$ denotes the $t$-fold iterate of $s$. Let $\mathcal W_t(n)$ be the set of $t$-stack-sortable permutations in $S_n$, and let $\mathcal W_t(n,k)=\{\pi\in\mathcal W_t(n):\des(\pi)=k\}$ and $\mathcal W_t(n,k,p)=\{\pi\in\mathcal W_t(n,k):\peak(\pi)=p\}$. Let
\[W_t(n)=|\mathcal W_t(n)|,\quad W_t(n,k)=|\mathcal W_t(n,k)|,\quad\text{and}\quad W_t(n,k,p)=|\mathcal W_t(n,k,p)|.\]
\end{definition}

Knuth simultaneously initiated the study of stack-sorting and the investigation of permutation patterns with the following theorem. 

\begin{theorem}[\!\!\cite{Knuth}]\label{Thm1}
A permutation is $1$-stack-sortable if and only if it avoids the pattern $231$. Furthermore, \[W_1(n)=|\Av_n(231)|=C_n,\] where $C_n=\frac{1}{n+1}{2n\choose n}$ is the $n^\text{th}$ Catalan number. 
\end{theorem} 

In his dissertation, West conjectured a formula for $W_2(n)$, which Zeilberger later proved. 

\begin{theorem}[\!\!\cite{Zeilberger}]\label{Thm2}
We have \[W_2(n)=\frac{2}{(n+1)(2n+1)}{3n\choose n}.\] 
\end{theorem}

Combinatorial proofs of Zeilberger's theorem emerged later in \cite{Cori,Dulucq,Dulucq2,Goulden}. Some authors have investigated the enumeration of $2$-stack-sortable permutations according to various statistics \cite{BonaSimplicial,Bousquet98,Bouvel,Dulucq}. The articles \cite{Duchi} and \cite{Fang} give different proofs that new combinatorial objects called ``fighting fish" are counted by the numbers $W_2(n)$. The authors of \cite{Bevan} studied what they called ``$n$-point dominoes," and they have found that there are $W_2(n+1)$ such objects. 

There is very little known about $t$-stack-sortable permutations when $t\geq 3$ is fixed. \'Ulfarsson \cite{Ulfarsson} characterized $3$-stack-sortable permutations in terms of new ``decorated patterns," but the characterization is too unwieldy to yield any additional information. The recent paper \cite{Albert} shows that for every $t\geq 1$, the set of $t$-stack-sortable permutations can be described by a sentence in a first-order logical theory that the authors call $\mathsf{TOTO}$. The paper \cite{Claessonn-4} also investigates $t$-stack-sortable permutations when $t=n-r$ for some fixed $r$ (focusing on the case in which $r=4$). For fixed $t\geq 3$, the best known general upper bound for $W_t(n)$ (see \cite[Theorem 3.4]{BonaSurvey}), is the estimate 
\begin{equation}\label{Eq1}
W_t(n)\leq(t+1)^{2n}.
\end{equation} 
The current author showed \cite{DefantPreimages} that 
\begin{equation}\label{Eq10}
\lim_{n\to\infty}W_3(n)^{1/n}<12.53296\quad\text{and}\quad \lim_{n\to\infty}W_4(n)^{1/n}<21.97225.
\end{equation}
The limits in \eqref{Eq10} are known to exist (see Section~\ref{Sec:Data}). Recently, B\'ona has obtained a new proof of the first inequality in \eqref{Eq10} using ``stack words." It also follows from Theorem~\ref{Thm2} that 
\begin{equation}\label{Eq2}
\lim_{n\to\infty}W_t(n)^{1/n}\geq 6.75\quad\text{for all}\quad t\geq 2.
\end{equation}
When $t\geq 3$, we refer to \eqref{Eq2} as a ``trivial" lower bound for the growth rate of $W_t(n)$, even though it relies on the highly nontrivial enumeration of $2$-stack-sortable permutations. Remarkably, \eqref{Eq2} was the best known lower bound for $\lim\limits_{n\to\infty}W_t(n)^{1/n}$ for all $t\geq 2$ until now. 

B\'ona \cite{BonaSymmetry} proved that the polynomial $\displaystyle\sum_{k=0}^{n-1}W_t(n,k)x^k=\sum_{\sigma\in \mathcal W_t(n)}x^{\des(\sigma)}$ is symmetric and unimodal (see Section~\ref{Sec:Symmetry} for the relevant definitions). In fact, his proof actually shows that $\displaystyle\sum_{\sigma\in s^{-1}(A)}x^{\des(\sigma)}$ is symmetric and unimodal for every set $A\subseteq S_n$. Br\"and\'en strengthened this result with the following theorem (we define $\gamma$-nonnegativity in Section~\ref{Sec:Symmetry}). 

\begin{theorem}[\!\!\cite{BrandenActions}]\label{Thm3}
If $A\subseteq S_n$, then \[\sum_{\sigma\in s^{-1}(A)}x^{\des(\sigma)}=\sum_{m=0}^{\left\lfloor\frac{n-1}{2}\right\rfloor}\frac{|\{\sigma\in s^{-1}(A):\peak(\sigma)=m\}|}{2^{n-1-2m}}x^m(1+x)^{n-1-2m}.\] In particular, $\displaystyle\sum_{\sigma\in s^{-1}(A)}x^{\des(\sigma)}$ is $\gamma$-nonnegative. 
\end{theorem}

In the present article, we concern ourselves with the following four conjectures of B\'ona. Recall that a sequence $(a_n)_{n\geq 1}$ of positive numbers is called \emph{log-convex} if $(a_{n+1}/a_n)_{n\geq 1}$ is nondecreasing. 

\begin{conjecture}[\!\!\cite{Bona,BonaSurvey}]\label{Conj3}
For all $n,t\geq 1$, we have \[W_t(n)\leq{(t+1)n\choose n}.\]
\end{conjecture}

\begin{conjecture}[\!\!\cite{BonaPrivate}]\label{Conj4}
For every $t\geq 1$, the sequence $(W_t(n))_{n\geq 1}$ is log-convex. 
\end{conjecture}

\begin{conjecture}[\!\!\cite{Elder}]\label{Conj6}
If $t$ is even, then $W_t(n)$ is frequently odd. If $t$ is odd, then $W_t(n)$ is rarely odd. 
\end{conjecture}

\begin{conjecture}[\!\!\cite{BonaSymmetry}]\label{Conj5}
For all $n,t\geq 1$, the polynomial $\displaystyle\sum_{\sigma\in \mathcal W_t(n)}x^{\des(\sigma)}$ has only real roots. 
\end{conjecture}

\begin{remark}\label{Rem4}
B\'ona's motivation for formulating Conjecture~\ref{Conj3} came from the idea of encoding elements of $\mathcal W_t(n)$ as $n$-uniform words over a $(t+1)$-element alphabet (see \cite{BonaWords} and \cite{BonaSurvey} for more details). His motivation behind Conjecture~\ref{Conj4} came from an observation that the sequences $(W_t(n))_{n\geq 1}$ appear to be similar to the sequences that enumerate principal permutation classes, which he has also conjectured are log-convex. For example, B\'ona has observed that his methods in \cite{BonaRational} can be used to show that for fixed $n,t\geq 1$, the number of $t$-stack-sortable permutations of length $n$ with $c$ components is monotonically decreasing as a function of $c$. Similarly, his methods allow one to prove that the generating functions $\sum_{n\geq 1}W_t(n)x^n$ are not rational. B\'ona formulated Conjecture~\ref{Conj5} after observing that it holds when $t=1$ and when $t=n-1$ (it also holds when $t\geq n$ because this is equivalent to the $t=n-1$ case). Br\"and\'en \cite{Branden3} proved this conjecture in the cases $t=2$ and $t=n-2$, but the remaining cases are still open.  

Conjecture~\ref{Conj6} requires some explanation. Using B\'ona's result that $\sum_{\sigma\in\mathcal W_t(n)}x^{\des(\sigma)}$ is symmetric, one can easily deduce that $W_t(n)$ is even whenever $n$ is even. Therefore, it is natural to consider the parity of $W_t(n)$ when $n$ is odd. Let $\mathfrak g_t(m)$ be the number of integers $n$ with $1\leq n\leq m$ such that $W_t(n)$ is odd. Let $F_r$ denote the $r^\text{th}$ Fibonacci number (with $F_1=F_2=1$). Using Theorems~\ref{Thm1} and \ref{Thm2}, one can show that $\mathfrak g_1(2^r)=r$ and $\mathfrak g_2(2^r)=F_r$ for all positive integers $r$. B\'ona \cite{Elder} interpreted this as saying $W_1(n)$ is rarely odd while $W_2(n)$ is frequently odd, and this led him to formulate Conjecture~\ref{Conj6}. One could formalize this by saying that $W_t(n)$ is \emph{rarely odd} if $\limsup\limits_{m\to\infty}\dfrac{\log\mathfrak g_t(m)}{\log m}=0$ and is \emph{frequently odd} $\liminf\limits_{m\to\infty}\dfrac{\log\mathfrak g_t(m)}{\log m}>0$ (although B\'ona did not use this formalism). B\'ona's motivation behind Conjecture~\ref{Conj6} also came from the idea of encoding $t$-stack-sortable permutations as words.  
\end{remark}

\subsection{Summary of Main Results}
In Section~\ref{Sec:DecompLemma}, we formulate a ``decomposition lemma," which provides a new method for analyzing preimages of permutations under the stack-sorting map. We actually prove a stronger lemma, which we call the refined decomposition lemma, that allows us to take the statistics $\des$ and $\peak$ into account. In Section~\ref{Sec:Formulas}, we briefly review some formulas arising from the theory of new combinatorial objects called ``valid hook configurations." In Section~\ref{Sec:2-Stack}, we use the decomposition lemma to give a new proof of Zeilberger's formula for $W_2(n)$. We also use the refined decomposition lemma to find an algebraic equation satisfied by the generating function of the numbers $W_2(n,k,p)$. This equation is new. 

Our new proof of Zeilberger's formula is the first one that generalizes to the setting of $3$-stack-sortable permutations. In Section~\ref{Sec:3-Stack}, we use the refined decomposition lemma to prove a recurrence relation for the numbers $W_3(n,k,p)$. Specializing this theorem gives us a recurrence for $W_3(n,k)$, and specializing further gives a recurrence for $W_3(n)$. This yields the first polynomial-time algorithm for computing $W_3(n)$. According to Wilf \cite{Wilf}, we have solved the problem of counting $3$-stack-sortable permutations. More precisely, he would say that we have ``$p$-solved" this problem. 

Before now, the values of $W_3(n)$ were only known up to $n=13$. Indeed, the only algorithm that was used to compute these numbers before now relied on a brute-force approach. Using our recurrence, we have generated the values of $W_3(n)$ for $n\leq 174$. We have added these terms to sequence A134664 in the Online Encyclopedia of Integer Sequences \cite{OEIS}. There are two significant theoretical implications of these computations. First, we will see in Section~\ref{Sec:Data} that B\'ona's Conjectures~\ref{Conj3} and \ref{Conj4} cannot both be true. Thus, we have disproven a conjecture of B\'ona, although we do not yet know with absolute certainty which one. Let us remark, however, that the data suggests very strongly that Conjecture~\ref{Conj4} is true while Conjecture~\ref{Conj3} is false. Furthermore, it appears that our recurrence coupled with sufficient computing time (and clever computing!) should allow one to completely disprove Conjecture~\ref{Conj3}. Second, we will prove that $\lim\limits_{n\to\infty}W_3(n)^{1/n}\geq 8.659702$; this is the first nontrivial lower bound for $\lim\limits_{n\to\infty}W_3(n)^{1/n}$. In Section~\ref{Sec:Lower}, we prove that $\lim\limits_{n\to\infty}W_t(n)^{1/n}\geq(\sqrt{t}+1)^2$ for every $t\geq 1$, yielding the first nontrivial lower bounds for these growth rates for all $t\geq 4$. As a corollary, we improve a result of Smith concerning permutations that can be sorted by $t$ stacks in series using the so-called ``left-greedy algorithm" \cite{Smith}. Although there are multiple ways one could rigorously interpret B\'ona's Conjecture~\ref{Conj6}, we will see in Section~\ref{Sec:Data} that every reasonable interpretation of the conjecture is likely to be false.  

We have also computed the numbers $W_3(n,k)$ for $n\leq 43$, allowing us to verify Conjecture~\ref{Conj5} when $t=3$ and $n\leq 43$ (see OEIS sequence A324916 \cite{OEIS}). In Section~\ref{Sec:Symmetry}, we show that the formulas from Section~\ref{Sec:Formulas} easily imply Br\"and\'en's Theorem~\ref{Thm3}. We also provide a two-element set $A\subseteq S_{11}$ such that $\sum_{\sigma\in s^{-1}(A)}x^{\des(\sigma)}$ is not real-rooted. This provides a negative answer to the last part of Question 12.1 in \cite{DefantClass}, which we interpret as a small amount of evidence against B\'ona's Conjecture~\ref{Conj5}. Section~\ref{Sec:Conclusion} concludes the paper with a new conjecture about $\lim\limits_{n\to\infty}W_3(n)^{1/n}$ and several conjectures about the numbers $\mathfrak g_3(m)$ (defined in Remark~\ref{Rem4}). 

Before we proceed, let us make one additional remark about the usefulness of the decomposition lemma that we prove in Section~\ref{Sec:DecompLemma}. In a subsequent paper \cite{DefantEnumeration}, we apply this lemma in order to settle several conjectures of the current author from \cite{DefantClass}. More precisely, we complete the project of determining $|s^{-1}(\Av_n(\tau^{(1)},\ldots,\tau^{(r)}))|$ for every subset $\{\tau^{(1)},\ldots,\tau^{(r)}\}\subseteq S_3$ with the exception of the singleton set $\{321\}$. This allows us to enumerate a new permutation class, find a new example of an unbalanced Wilf equivalence, and prove a conjecture of Hossain concerning the so-called ``Boolean-Catalan numbers." Hence, one can even view the decomposition lemma as a bridge that allows one to use the stack-sorting map $s$ as a tool for proving results that were conjectured without any reference to stack-sorting. 

\section{The Decomposition Lemma}\label{Sec:DecompLemma}

West \cite{West} defined the \emph{fertility} of a permutation $\pi$ to be $|s^{-1}(\pi)|$, the number of preimages of $\pi$ under $s$. He then performed extensive calculations in order to compute the fertilities of the permutations of the forms \[23\cdots k1(k+1)\cdots n,\quad 12\cdots(k-2)k(k-1)(k+1)\cdots n,\quad\text{and}\quad k12\cdots(k-1)(k+1)\cdots n.\] Bousquet-M\'elou \cite{Bousquet} found a method for determining whether or not a given permutation is \emph{sorted}, meaning that its fertility is positive. She then asked for a general method for computing the fertility of any given permutation. The current author achieved this in even greater generality in \cite{DefantPostorder, DefantPreimages, DefantClass} using new combinatorial objects called ``valid hook configurations." In this section, we prove the refined decomposition lemma and the decomposition lemma, which provide a new method for analyzing fertilities of permutations.  

The \emph{plot} of a permutation $\pi=\pi_1\cdots\pi_n$ is the figure showing the points $(i,\pi_i)$ for all $i\in[n]$. For example, the image on the left in Figure~\ref{Fig2} is the plot of $3142567$. A \emph{hook} of $\pi$ is obtained by starting at a point $(i,\pi_i)$ in the plot of $\pi$, drawing a vertical line segment moving upward, and then drawing a horizontal line segment to the right that connects with a point $(j,\pi_j)$. In order for this to make sense, we must have $i<j$ and $\pi_i<\pi_j$. The point $(i,\pi_i)$ is called the \emph{southwest endpoint} of the hook, while $(j,\pi_j)$ is called the \emph{northeast endpoint}. Let $\SW_i(\pi)$ be the set of hooks of $\pi$ with southwest endpoint $(i,\pi_i)$. The right image in Figure~\ref{Fig2} shows a hook of $3142567$. This hook is in $\SW_3(3142567)$ because its southwest endpoint is $(3,4)$. 

\begin{figure}[h]
  \centering
  \subfloat[]{\includegraphics[width=0.17\textwidth]{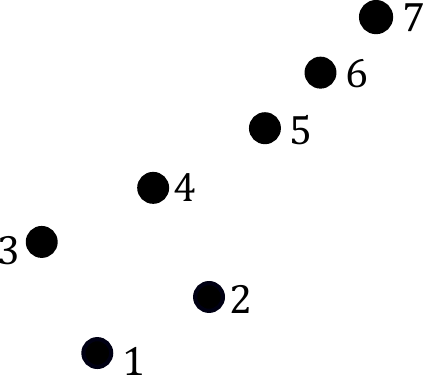}}
  \hspace{1.5cm}
  \subfloat[]{\includegraphics[width=0.17\textwidth]{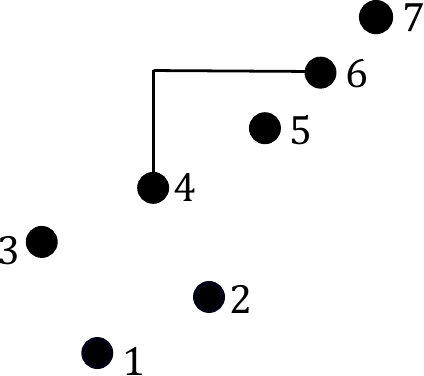}}
  \caption{The left image is the plot of $3142567$. The right image shows this plot along with a single hook.}\label{Fig2}
\end{figure}

Define the \emph{tail length} of a permutation $\pi=\pi_1\cdots\pi_n\in S_n$, denoted $\tl(\pi)$, to be the smallest
nonnegative integer $\ell$ such that $\pi_{n-\ell}\neq n-\ell$. We make the convention that $\tl(1\cdots n)=n$. The \emph{tail} of $\pi$ is the sequence of points $(n-\tl(\pi)+1,n-\tl(\pi)+1),\ldots,(n,n)$ in the plot of $\pi$. For example, the tail length of the permutation $3142567$ shown in Figure~\ref{Fig2} is $3$, and the tail of this permutation is $(5,5),(6,6),(7,7)$. We say a descent $d$ of $\pi$ is \emph{tail-bound} if every hook in $\SW_d(\pi)$ has its northeast endpoint in the tail of $\pi$. The only tail-bound descent of $3142567$ is $3$.  

Suppose $H$ is a hook of a permutation $\pi=\pi_1\cdots\pi_n$ with southwest endpoint $(i,\pi_i)$ and northeast endpoint $(j,\pi_j)$. Let $\pi_U^H=\pi_1\cdots\pi_i\pi_{j+1}\cdots\pi_n$ and $\pi_S^H=\pi_{i+1}\cdots\pi_{j-1}$. The permutations $\pi_U^H$ and $\pi_S^H$ are called the \emph{$H$-unsheltered subpermutation of $\pi$} and the \emph{$H$-sheltered subpermutation of $\pi$}, respectively. For example, if $\pi=3142567$ and $H$ is the hook shown on the right in Figure~\ref{Fig2}, then $\pi_U^H=3147$ and $\pi_S^H=25$. In all of the cases we consider in this paper, the plot of $\pi_S^H$ lies completely below the hook $H$ in the plot of $\pi$ (it is ``sheltered" by the hook $H$). 

\begin{lemma}[Refined Decomposition Lemma]\label{Lem1}
If $d$ is a tail-bound descent of a permutation $\pi\in S_n$, then \[\sum_{\sigma\in s^{-1}(\pi)}x^{\des(\sigma)+1}y^{\peak(\sigma)+1}\] \[=\sum_{H\in\SW_d(\pi)}\left(\sum_{\mu\in s^{-1}(\pi_U^H)}x^{\des(\mu)+1}y^{\peak(\mu)+1}\right)\left(\sum_{\lambda\in s^{-1}(\pi_S^H)}x^{\des(\lambda)+1}y^{\peak(\lambda)+1}\right).\]
\end{lemma}

\begin{proof}
If the tail of $\pi$ is empty, then both sides of the desired equation are $0$ because $s^{-1}(\pi)$ and $\SW_d(\pi)$ are empty. Hence, we may assume $\tl(\pi)\geq 1$. Let $a=\pi_d$. Given $\sigma\in s^{-1}(\pi)$, we let $f_\sigma$ be the entry that forces $a$ to leave the stack when we apply the stack-sorting procedure (described in the introduction) to $\sigma$. More precisely, $f_\sigma$ is the leftmost entry that appears to the right of $a$ in $\sigma$ and is larger than $a$. Note that $f_\sigma$ appears to the right of $a$ in $\pi$. Because $d$ is tail-bound, this means that the point $(f_\sigma,f_\sigma)$ is in the tail of $\pi$. Given a point $(j,j)$ in the tail of $\pi$, let $E_j$ be the set of permutations $\sigma\in s^{-1}(\pi)$ such that $f_\sigma=j$.    

Now fix a point $(j,j)$ in the tail of $\pi$, and let $H$ be the hook in $\SW_d(\pi)$ with northeast endpoint $(j,j)$. We will show that \[\sum_{\sigma\in E_j}x^{\des(\sigma)+1}y^{\peak(\sigma)+1}=\left(\sum_{\mu\in s^{-1}(\pi_U^H)}x^{\des(\mu)+1}y^{\peak(\mu)+1}\right)\left(\sum_{\lambda\in s^{-1}(\pi_S^H)}x^{\des(\lambda)+1}y^{\peak(\lambda)+1}\right),\] from which the lemma will follow. We can write $\pi=L\,a\,\pi_S^H\,j\,R$, where $L=\pi_1\cdots\pi_{d-1}$ and $R=(j+1)\cdots n$. Suppose $\sigma\in E_j$. Let us write $\sigma=\tau\,j\,\tau'$. Because $j=f_\sigma$, it follows from the stack-sorting procedure that every entry in $\tau$ that is smaller than $a$ must appear to the left of $a$ in $s(\sigma)=\pi$. This implies that every entry in $\pi_S^H$ that is smaller than $a$ is in $\tau'$. In particular, $\pi_{d+1}$ is in $\tau'$ (we know that $a>\pi_{d+1}$ because $d$ is a descent of $\pi$). Now suppose $b$ is an entry in $\pi_S^H$ that is larger than $a$. If $b$ is in $\tau$, then we can appeal to the stack-sorting procedure again to see that $b$ must appear to the left of $\pi_{d+1}$ in $\pi$. This is impossible, so every entry in $\pi_S^H$ is in $\tau'$. The stack-sorting procedure forces every entry in $L$ to be in $\tau$, so every entry in $\tau'$ that is not in $\pi_S^H$ must be an entry in $R$. Furthermore, an entry in $R$ that is also in $\tau'$ cannot appear to the left of one of the entries from $\pi_S^H$ in $\tau'$ (otherwise, $j$ would appear to the right of one of the entries from $\pi_S^H$ in $\pi$). This proves that we can write $\tau'=\lambda\tau''$, where $\lambda$ is a permutation of the entries in $\pi_S^H$. Moreover, every entry in $\tau''$ is in $R$. 

Now let $\mu=\tau\tau''$. One can verify that $s(\mu)=\pi_U^H$ and $s(\lambda)=\pi_S^H$. Let $\delta=1$ if $1$ is a descent of $\tau''$, and let $\delta=0$ otherwise. Because $j=f_\sigma$, the leftmost entry in $\tau''$ is the leftmost entry in $\mu$ that appears to the right of $a$ in $\mu$ and is larger than $a$ (if no such entry exists, then $\tau''$ is empty). Also, the rightmost entry in $\tau$ is less than $j$. Combining these observations, we find that $\des(\sigma)+1=\des(\tau)+1+\des(\lambda)+\des(\tau'')+1=\des(\mu)+1+\des(\lambda)+1$ and $\peak(\sigma)+1=\peak(\tau)+1+\peak(\lambda)+\peak(\tau'')+\delta+1=\peak(\mu)+1+\peak(\lambda)+1$. 

We have shown how to take a permutation $\sigma\in E_j$ and decompose it into permutations $\mu\in s^{-1}(\pi_U^H)$ and $\lambda\in s^{-1}(\pi_S^H)$ with $\des(\sigma)+1=\des(\mu)+1+\des(\lambda)+1$ and $\peak(\sigma)+1=\peak(\mu)+1+\peak(\lambda)+1$. We can easily reverse this procedure. Namely, if we are given $\mu$ and $\lambda$, we can write $\mu=\tau\tau''$ so that the leftmost entry in $\tau''$ is the leftmost entry in $\mu$ that appears to the right of $a$ in $\mu$ and is larger than $a$. We then recover $\sigma$ by letting $\sigma=\tau\,j\,\lambda\,\tau''$.  
\end{proof}

\begin{corollary}[Decomposition Lemma]\label{Cor1}
If $d$ is a tail-bound descent of a permutation $\pi\in S_n$, then \[|s^{-1}(\pi)|=\sum_{H\in\SW_d(\pi)}|s^{-1}(\pi_U^H)|\cdot|s^{-1}(\pi_S^H)|.\]
\end{corollary}
\begin{proof}
Set $x=y=1$ in Lemma~\ref{Lem1}. 
\end{proof}

\section{Fertility Formulas}\label{Sec:Formulas}
The purpose of this brief section is to establish some terminology and state some formulas from \cite{DefantPostorder} that we will use in Section~\ref{Sec:Symmetry}. We will also use a very special consequence of Theorem~\ref{Thm10} in Section~\ref{Sec:2-Stack} when we analyze the generating function of the numbers $W_2(n,k,p)$. 

A \emph{composition of $b$ into $a$ parts} is an $a$-tuple of positive integers that sum to $b$. For example, $(3,4,3,1)$ is a composition of $11$ into $4$ parts. Let $\Comp_a(b)$ denote the set of compositions of $b$ into $a$ parts. Let $C_r=\frac{1}{r+1}{2r\choose r}$ denote the $r^\text{th}$ Catalan number. Let 
\begin{equation}\label{Eq33}
N(r,i)=\frac 1r{r\choose i}{r\choose i-1}\quad \text{and}\quad V(r,j)=2^{r-2j+1}{r-1\choose 2j-2}C_{j-1}.
\end{equation}
Let 
\begin{equation}\label{Eq34}
N_r(x)=\sum_{i=1}^rN(r,i)x^i\quad\text{and}\quad V_r(y)=\sum_{j=1}^rV(r,j)y^j.
\end{equation} The numbers $N(r,i)$ are called \emph{Narayana numbers}. They are given in the OEIS sequence A001263 and constitute the most common refinement of the Catalan numbers \cite{OEIS}. The polynomials $N_r(x)$ are called \emph{Narayana polynomials}. Among many other things, the Narayana numbers $N(r,i)$ count binary plane trees with $r$ vertices and $i-1$ right edges. The numbers $V(r,j)$, which count binary plane trees with $r$ vertices and $j$ leaves, are given in the OEIS sequence A091894. Let $L(r,i,j)$ be the number of binary plane trees with $r$ vertices, $i-1$ right edges, and $j$ leaves. Letting $F(w,x,y)=\sum_{r,i,j\geq 0}L(r,i,j)w^rx^iy^j$, we have 
\begin{equation}\label{Eq15}
F(w,x,y)=x + wxy + w(F(w,x,y)+1)(F(w,x,y)-x).
\end{equation} This yields \[F(w,x,y)=\frac{1-w+wx-\sqrt{(1-w+wx)^2-4wx(1-w+wy)}}{2w},\] from which one obtains 
\begin{equation}\label{Eq36}
L(r,i,j)=\frac{1}{r+1-j}{r-1\choose r-j}{r+1-j\choose j}{r+1-2j\choose i-j}.
\end{equation}
Let 
\begin{equation}\label{Eq19}
L_r(x,y)=\sum_{i=1}^r\sum_{j=1}^rL(r,i,j)x^iy^j
\end{equation} so that \[L_r(x,1)=N_r(x)\quad\text{and}\quad L_r(1,y)=V_r(y).\] 

\begin{theorem}[\!\!\cite{DefantPostorder}\footnote{Strictly speaking, the first statement in Theorem~\ref{Thm10} has not been stated explicitly before. However, the proofs of Corollary 5.1 and Theorem 5.2 in \cite{DefantPostorder} immediately generalize to yield that statement.}]\label{Thm10} 
If $n\geq 1$ and $\pi=\pi_1\cdots\pi_n$ has exactly $k$ descents, then there exists a set $\mathcal V(\pi)\subseteq\Comp_{k+1}(n-k)$ such that 
\begin{equation}\label{Eq11}
\sum_{\sigma\in s^{-1}(\pi)}x^{\des(\sigma)+1}y^{\peak(\sigma)+1}=\sum_{(q_0,\ldots,q_k)\in\mathcal V(\pi)}\prod_{t=0}^kL_{q_t}(x,y).
\end{equation} In particular, 
\begin{equation}\label{Eq12}
\sum_{\sigma\in s^{-1}(\pi)}x^{\des(\sigma)+1}=\sum_{(q_0,\ldots,q_k)\in\mathcal V(\pi)}\prod_{t=0}^kN_{q_t}(x)
\end{equation}
and 
\begin{equation}\label{Eq13}\sum_{\sigma\in s^{-1}(\pi)}y^{\peak(\sigma)+1}=\sum_{(q_0,\ldots,q_k)\in\mathcal V(\pi)}\prod_{t=0}^kV_{q_t}(y).
\end{equation} Thus, 
\begin{equation}\label{Eq14}
|s^{-1}(\pi)|=\sum_{(q_0,\ldots,q_k)\in\mathcal V(\pi)}\prod_{t=0}^kC_{q_t}.
\end{equation}
\end{theorem}

\begin{remark}\label{Rem1}
If $\pi=123\cdots n$, then the above theorem, along with Theorem~\ref{Thm1}, tells us that \[\sum_{\sigma\in s^{-1}(123\cdots n)}x^{\des(\sigma)+1}y^{\peak(\sigma)+1}=\sum_{\sigma\in\Av_n(231)}x^{\des(\sigma)+1}y^{\peak(\sigma)+1}=L_n(x,y).\]
\end{remark}

\section{A New Proof of the Formula for $W_2(n)$}\label{Sec:2-Stack} 

Recall from Section~\ref{Sec:DecompLemma} the definition of the tail length $\tl(\pi)$ of a permutation $\pi$. Let $B_\ell(n)$ (respectively, $B_{\geq\ell}(n)$) be the number of $2$-stack-sortable permutations $\sigma\in \mathcal W_2(n+\ell)$ such that $\tl(s(\sigma))=\ell$ (respectively, $\tl(s(\sigma))\geq\ell$). Let \[\mathcal D_\ell(n)=\{\pi\in\Av_{n+\ell}(231):\tl(\pi)=\ell\}\quad\text{and}\quad\mathcal D_{\geq\ell}(n)=\{\pi\in\Av_{n+\ell}(231):\tl(\pi)\geq\ell\}.\] Because $\mathcal W_2(n)=s^{-1}(\mathcal W_1(n))=s^{-1}(\Av_{n}(231))$ by Theorem~\ref{Thm1}, we can write \[B_\ell(n)=|s^{-1}(\mathcal D_\ell(n))|\quad\text{and}\quad B_{\geq \ell}(n)=|s^{-1}(\mathcal D_{\geq \ell}(n))|.\] 

Suppose $\pi\in\mathcal D_\ell(n+1)$ is such that $\pi_{n+1-i}=n+1$ (where $n\geq 0$). Then $n+1-i$ is a tail-bound descent of $\pi$. 
The decomposition lemma (Corollary~\ref{Cor1}) tells us that $|s^{-1}(\pi)|$ is equal to the number of triples $(H,\mu,\lambda)$, where $H\in\SW_{n+1-i}(\pi)$, $\mu\in s^{-1}(\pi_U^H)$, and $\lambda\in s^{-1}(\pi_S^H)$. Choosing $H$ amounts to choosing the number $j\in\{1,\ldots,\ell\}$ such that the northeast endpoint of $H$ is $(n+1+j,n+1+j)$. The permutation $\pi$ and the choice of $H$ determine the permutations $\pi_U^H$ and $\pi_S^H$. On the other hand, the choices of $H$ and the permutations $\pi_U^H$ and $\pi_S^H$ uniquely determine $\pi$. It follows that $B_\ell(n+1)$, which is the number of ways to choose an element of $s^{-1}(\mathcal D_\ell(n+1))$, is also the number of ways to choose the tuple $(j,\pi_U^H,\pi_S^H,\mu,\lambda)$. Let us fix a choice of $j$. 

Because $\pi$ avoids $231$, $\pi_U^H$ must be a permutation of the set $\{1,\ldots,n-i\}\cup\{n+1\}\cup\linebreak\{n+2+j,\ldots,n+\ell+1\}$, while $\pi_S^H$ must be a permutation of $\{n-i+1,\ldots,n+j\}\setminus\{n+1\}$. Therefore, choosing $\pi_U^H$ and $\pi_S^H$ is equivalent to choosing their normalizations. The normalization of $\pi_U^H$ is in $\mathcal D_{\geq \ell-j+1}(n-i)$, while the normalization of $\pi_S^H$ is in $\mathcal D_{\geq j-1}(i)$ (see Figure~\ref{Fig1}). Any element of $\mathcal D_{\geq \ell-j+1}(n-i)$ can be chosen as the normalization of $\pi_U^H$, and any element of $\mathcal D_{\geq j-1}(i)$ can be chosen as the normalization of $\pi_S^H$. Also, $\pi_U^H$ and $\pi_S^H$ have the same fertilities as their normalizations. Combining these facts, we find that the number of choices for the pair $(\pi_U^H,\mu)$ is $|s^{-1}(\mathcal D_{\geq \ell-j+1}(n-i))|=B_{\geq \ell-j+1}(n-i)$. Similarly, the number of choices for the pair $(\pi_S^H,\lambda)$ is $B_{\geq j-1}(i)$. Hence, 
\begin{equation}\label{Eq3}
B_\ell(n+1)=\sum_{i=1}^n\sum_{j=1}^\ell B_{\geq \ell-j+1}(n-i)B_{\geq j-1}(i).
\end{equation}

\begin{figure}[h]
\begin{center}
\includegraphics[width=.4\linewidth]{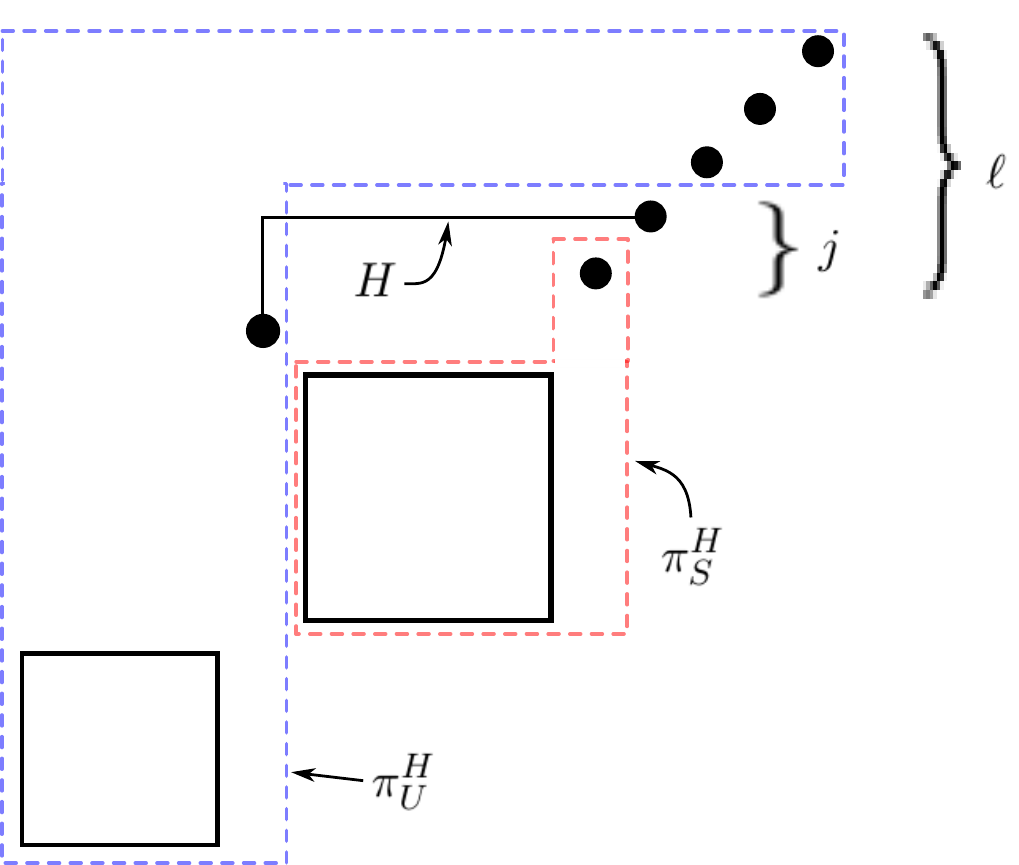}
\caption{The decomposition of $\pi$ into $\pi_U^H$ and $\pi_S^H$.}
\label{Fig1}
\end{center}  
\end{figure}

Let \[G_\ell(w)=\sum_{n\geq 0}B_{\geq\ell}(n)w^n\quad\text{and}\quad I(w,z)=\sum_{\ell\geq 0}G_\ell(w)z^\ell.\] Note that \[G_\ell(0)=B_{\geq \ell}(0)=|s^{-1}(\mathcal D_{\geq \ell}(0))|=|s^{-1}(123\cdots\ell)|=C_\ell\] by Theorem~\ref{Thm1}. 
Let $C(z)=\sum_{n\geq 0}C_nz^n=\dfrac{1-\sqrt{1-4z}}{2z}$ be the generating function of the Catalan numbers. Because $B_{\geq 0}(n)=W_2(n)$ is the total number of $2$-stack-sortable permutations in $S_n$, our goal is to understand the generating function \[I(w,0)=G_0(w)=\sum_{n\geq 0}B_{\geq 0}(n)w^n=\sum_{n\geq 0}W_2(n)w^n.\] 

By \eqref{Eq3}, we have \[\sum_{\ell\geq 0}\sum_{n\geq 0}B_\ell(n+1)w^nz^\ell=\sum_{\ell\geq 0}\sum_{j=1}^\ell \sum_{n\geq 0}\sum_{i=1}^n B_{\geq \ell-j+1}(n-i)B_{\geq j-1}(i)w^nz^\ell\] \[=\sum_{\ell\geq 0}\sum_{j=1}^\ell G_{\ell-j+1}(w)(G_{j-1}(w)-G_{j-1}(0))z^\ell=\sum_{\ell\geq 0}\sum_{j=1}^\ell G_{\ell-j+1}(w)(G_{j-1}(w)-C_{j-1})z^\ell\] 
\begin{equation}\label{Eq4}
=\left(\sum_{r\geq 0}G_{r+1}(w)z^r\right)\left(\sum_{j\geq 1}(G_{j-1}(w)-C_{j-1})z^j\right)=(I(w,z)-I(w,0))(I(w,z)-C(z)).
\end{equation}

On the other hand, \[B_\ell(n+1)=B_{\geq\ell}(n+1)-B_{\geq \ell+1}(n),\] so \[\sum_{\ell\geq 0}\sum_{n\geq 0}B_\ell(n+1)w^nz^\ell=\sum_{\ell\geq 0}\sum_{n\geq 0}B_{\geq\ell}(n+1)w^nz^\ell-\sum_{\ell\geq 0}\sum_{n\geq 0}B_{\geq \ell+1}(n)w^nz^\ell\]
\begin{equation}\label{Eq5}
=\frac{1}{w}\sum_{\ell\geq 0}(G_\ell(w)-C_\ell)z^\ell-\frac{1}{z}\sum_{\ell\geq 0}G_{\ell+1}(w)z^{\ell+1}=\frac{I(w,z)-C(z)}{w}-\frac{I(w,z)-I(w,0)}{z}.
\end{equation} 
Combining \eqref{Eq4} and \eqref{Eq5} yields the equation 
\begin{equation}\label{Eq6}
(I(w,z)-I(w,0))(I(w,z)-C(z))=\frac{I(w,z)-C(z)}{w}-\frac{I(w,z)-I(w,0)}{z}.
\end{equation}
We can now solve \eqref{Eq6} for $C(z)$, use the standard Catalan functional equation $zC(z)^2+1-C(z)=0$, and clear denominators to obtain a polynomial \[\begin{split}
Q(u,v,w,z)=\: &-v w + z + 2 v w z + 
   v^2 w^2 z + (w - z - 2 w z - 2 v w^2 z + v^2 w^2 z) u  \\
 &+ (w^2 z - 
    2 v w^2 z + z^2 + 2 v w z^2 + v^2 w^2 z^2) u^2 + (w^2 z - 
    2 w z^2 - 2 v w^2 z^2) u^3 + w^2 z^2 u^4
\end{split}
\] such that 
\begin{equation}\label{Eq37}
Q(I(w,z),I(w,0),w,z)=0.
\end{equation} 

Let $Q_u'=\dfrac{\partial}{\partial u}Q(u,v,w,z)$. There is a unique fractional power series (Puiseux series) $Z=Z(w)$ such that $Z(w)=w+O(w^2)$ and 
\begin{equation}\label{Eq7}
Q_u'(I(w,Z),I(w,0),w,Z)=0.
\end{equation} Indeed, we can compute the coefficients of $Z(w)$ one at a time from the equation \eqref{Eq7} after we have initially computed sufficiently many terms of $I(w,z)$ via its combinatorial definition. Let $\Delta_uQ(v,w,z)$ be the discriminant of $Q(u,v,w,z)$ with respect to the variable $u$. A computer can explicitly compute this discriminant as $\Delta_uQ(v,w,z)=w^6 (1 - 4 z)^2 z^3 \widehat Q(v,w,z)$, where \[\begin{split}
\widehat Q(v,w,z)=\: &z^3 + 2 w z^2 (-3 + 2 v z) + w^4 z (1 + v + v^2 z)^2 + 
 w^2 z (9 + (2 - 10 v) z + 6 v^2 z^2) \\
 &+ 
 2 w^3 (-2 + (5 - 3 v) z - (-2 + v) v z^2 + 2 v^3 z^3).
\end{split}
\] At this point, we use Theorem 14\footnote{In the notation of \cite{Bousquet3}, we are applying Theorem 14 with $k=1$. Our polynomial $Q(u,v,w,z)$, power series $I(w,z)$, and power series $I(w,0)$ are playing the roles of $P(x_0,\ldots,x_k,t,v)$, $F(t,u)$, and $F_1(t)$, respectively, from that article.} from the paper \cite{Bousquet3}, which allows us to deduce from \eqref{Eq37} and \eqref{Eq7} that $z=Z(w)$ is a repeated root of $\Delta_uQ(I(w,0),w,z)$. Since $Z(w)=w+O(w^2)$, we know that $w^6(1-4Z)^2Z^3\neq 0$. Therefore, $z=Z(w)$ is a repeated root of $\widehat Q(I(w,0),w,z)$. The discriminant of a polynomial with a repeated root must be $0$. This means that $\Delta_z\widehat Q(I(w,0),w)=0$, where $\Delta_z\widehat Q(v,w)$ is the discriminant of $\widehat Q(v,w,z)$ with respect to $z$. Computing $\Delta_z\widehat Q(v,w)$ explicitly and ignoring extraneous factors, we find that $\mathcal R(I(w,0),w)=0$, where \[\mathcal R(v,w)=-1 + 11 w + w^2 + v^3 w^2 + v^2 w (2 + 3 w) + v (1 - 14 w + 3 w^2).\] 

To complete our new proof of Theorem~\ref{Thm2}, we follow the proof of Proposition 5.2 in \cite{Bousquet}. Namely, we consider the power series $U(w)$ defined by $U(w)=w(1+U(w))^3$. We then verify that $\mathcal R(1+U(w)-U(w)^2,w)=0$ and deduce that $I(w,0)=1+U(w)-U(w)^2$. Lagrange inversion then completes the proof that \[I(w,0)=\sum_{n\geq 0}\frac{2}{(n+1)(2n+1)}{3n\choose n}w^n.\] 

The above argument generalizes as follows. Let \[I_{x,y}(w,z)=\sum_{\ell\geq 0}\sum_{n\geq 0}\sum_{\sigma\in s^{-1}(\mathcal D_{\geq\ell}(n))}x^{\des(\sigma)+1}y^{\peak(\sigma)+1}w^nz^\ell.\] Note that $I_{x,y}(w,0)=\sum_{n\geq 0}\sum_{\sigma\in\mathcal W_2(n)}w^nx^{\des(\sigma)+1}y^{\peak(\sigma)+1}$. We make the convention that the empty permutation has $0$ descents and $-1$ peaks so that $I_{x,y}(0,0)=x$. Let $F$ be the generating function in \eqref{Eq15}. If we replace $B_\ell(n)$ with $\sum_{\sigma\in s^{-1}(\mathcal D_\ell(n))}x^{\des(\sigma)+1}y^{\peak(\sigma)+1}$, replace $B_{\geq \ell}(n)$ with $\sum_{\sigma\in s^{-1}(\mathcal D_{\geq\ell}(n))}x^{\des(\sigma)+1}y^{\peak(\sigma)+1}$, use the refined decomposition lemma instead of the decomposition lemma, and use Remark~\ref{Rem1} instead of Theorem~\ref{Thm1}, then the above argument produces the equation 
\[(I_{x,y}(w,z)-I_{x,y}(w,0))(I_{x,y}(w,z)-F(z,x,y))=\frac{I_{x,y}(w,z)-F(z,x,y)}{w}-\frac{I_{x,y}(w,z)-I_{x,y}(w,0)}{z}
\]
in place of \eqref{Eq6}. We then continue the argument, using the functional equation \eqref{Eq15} instead of the Catalan functional equation for $C(z)$, in order to arrive at the following theorem concerning the generating function of the numbers $W_2(n,k,p)$. 

\begin{theorem}\label{Thm7}
The generating function \[I_{x,y}(w,0)=\sum_{n\geq 0}\sum_{\sigma\in\mathcal W_2(n)}w^nx^{\des(\sigma)+1}y^{\peak(\sigma)+1}\] for the numbers $W_2(n,k,p)$ satisfies the equation $R(I_{x,y}(w,0),w,x,y)=0$, where \[R(v,w,x,y)=-x + (4 x + 8 x^2 - x y) w + (-6 x - 16 x^2 - 16 x^3 + 3 x y + 
    36 x^2 y) w^2 + (4 x + 8 x^2 - 3 x y - 36 x^2 y \] \[+ 
    27 x^2 y^2) w^3 + (-x + x y) w^4+(1 + (-4 - 12 x) w + (6 + 20 x + 32 x^2 - 33 x y) w^2 + (-4 - 4 x + 
    16 x^2 + 30 x y \] \[- 36 x^2 y) w^3 + (1 - 4 x + 3 x y) w^4)v +(4 w + (-4 - 22 x) w^2 + (-4 - 20 x + 8 x^2 + 33 x y) w^3 + (4 - 6 x + 
    3 x y) w^4)v^2\] \[+(6 w^2 + (4 - 12 x) w^3 + (6 - 4 x + x y) w^4)v^3+(4 w^3 + (4 - x) w^4)v^4+w^4v^5.\]  
\end{theorem}

\section{$3$-Stack-Sortable Permutations}\label{Sec:3-Stack} 

In the previous section, we counted $2$-stack-sortable permutations by viewing them as preimages of $231$-avoiding permutations under the stack-sorting map. In doing so, we had to keep track of the tail lengths of the $231$-avoiding permutations under consideration. In this section, we count $3$-stack-sortable permutations by viewing them as preimages of $2$-stack-sortable permutations. We will again keep track of tail lengths, but we will also need an additional new statistic. 

\begin{definition}\label{Def4}
Given $\pi=\pi_1\cdots\pi_n\in S_n$ and $a\in\{0,\ldots,n\}$, we say the open interval $(a,a+1)$ is a \emph{legal space for $\pi$} if there do not exist indices $i_1<i_2<i_3$ such that $\pi_{i_3}\leq a<\pi_{i_1}<\pi_{i_2}$. Let $\leg(\pi)$ be the number of legal spaces of $\pi$.  
\end{definition}  

For example, if $\pi\in S_n$, then $\leg(\pi)=n+1$ if and only if $\pi$ avoids $231$. The legal spaces of $145326$ are $(0,1),(1,2),(4,5),(5,6),(6,7)$, so $\leg(145326)=5$. Imagine adding a new point somewhere to the left of all points in the plot of a permutation $\pi$. One can think of the legal spaces of $\pi$ as the vertical positions where the new point can be inserted so as to not form a new $2341$ pattern. This is relevant for us because of the following characterization of $2$-stack-sortable permutations due to West. 

\begin{theorem}[\!\!\cite{West}]\label{Thm8}
A permutation is $2$-stack-sortable if and only if it avoids the pattern $2341$ and also avoids any $3241$ pattern that is not part of a $35241$ pattern. 
\end{theorem}  

We are now in a position to state and prove the main theorems of this article. In what follows, let $B_{\geq\ell}^{(g)}(n)$ be the number of $3$-stack-sortable permutations $\sigma\in \mathcal W_3(n+\ell)$ such that $\tl(s(\sigma))\geq\ell$ and $\leg(s(\sigma))=\ell+g$. Also, recall the definitions from Section~\ref{Sec:DecompLemma}. 

\begin{theorem}\label{Thm6}
If $n\geq 1$, then \[W_3(n)=\sum_{g=1}^{n+1}B_{\geq 0}^{(g)}(n),\] where the numbers $B_{\geq\ell}^{(g)}(n)$ satisfy the following relations. We have $B_{\geq\ell}^{(0)}(n)=0$ and \[B_{\geq\ell}^{(g)}(1)=\begin{cases} 0, & \mbox{if } g\neq 2; \\ C_{\ell+1}, & \mbox{if } g=2. \end{cases}\] If $n,g\geq 1$ and $\ell\geq 0$, then \[
B_{\geq\ell}^{(g)}(n+1)=\sum_{j=1}^\ell\left( \sum_{a=2}^{n}\sum_{b=\max\{2,g-a\}}^{g-1}\sum_{i=a-1}^{n-b+1} B_{\geq j-1}^{(a)}(i) B_{\geq \ell-j+1}^{(b)}(n-i)+ B_{\geq j-1}^{(g-1)}(n)C_{\ell-j+1}\right) + B_{\geq\ell+1}^{(g-1)}(n).\] 
\end{theorem}

\begin{proof}
The first statement and the fact that $B_{\geq\ell}^{(0)}(n)=0$ are clear from the definitions we have given. The permutations $\sigma$ counted by $B_{\geq\ell}^{(g)}(1)$ are in $S_{\ell+1}$ and satisfy $\tl(s(\sigma))\geq\ell$, so they must actually satisfy $s(\sigma)=123\cdots(\ell+1)$. Since $\leg(123\cdots(\ell+1))=\ell+2$, the formula for $B_{\geq\ell}^{(g)}(1)$ follows from Theorem~\ref{Thm1}. 

Now, let $B_\ell^{(g)}(n)$ be the number of $3$-stack-sortable permutations $\sigma\in \mathcal W_3(n+\ell)$ such that $\tl(s(\sigma))=\ell$ and $\leg(s(\sigma))=\ell+g$. Let 
\begin{equation}\label{Eq17}
\mathcal D_\ell^{(g)}(n)=\{\pi\in\mathcal W_2(n+\ell):\tl(\pi)=\ell,\,\leg(\pi)=\ell+g\}
\end{equation} and 
\begin{equation}\label{Eq18}
\mathcal D_{\geq\ell}^{(g)}(n)=\{\pi\in\mathcal W_2(n+\ell):\tl(\pi)\geq\ell,\,\leg(\pi)=\ell+g\}
\end{equation} so that \[B_\ell^{(g)}(n)=|s^{-1}(\mathcal D_\ell^{(g)}(n))|\quad\text{and}\quad B_{\geq\ell}^{(g)}(n)=|s^{-1}(\mathcal D_{\geq\ell}^{(g)}(n))|.\] We have $B_{\geq \ell}^{(g)}(n+1)=B_\ell^{(g)}(n+1)+B_{\geq\ell+1}^{(g-1)}(n)$, so we need to show that 
\begin{equation}\label{Eq16}
B_{\ell}^{(g)}(n+1)=\sum_{a=2}^{n}\sum_{b=\max\{2,g-a\}}^{g-1}\sum_{i=a-1}^{n-b+1}\sum_{j=1}^\ell B_{\geq j-1}^{(a)}(i) B_{\geq \ell-j+1}^{(b)}(n-i)+\sum_{j=1}^\ell B_{\geq j-1}^{(g-1)}(n)C_{\ell-j+1}. 
\end{equation}

Suppose $\pi\in\mathcal D_\ell^{(g)}(n+1)$ is such that $\pi_{n+1-i}=n+1$ (where $n\geq 0$). The decomposition lemma (Corollary~\ref{Cor1}) tells us that $|s^{-1}(\pi)|$ is equal to the number of triples $(H,\mu,\lambda)$, where $H\in\SW_{n+1-i}(\pi)$, $\mu\in s^{-1}(\pi_U^H)$, and $\lambda\in s^{-1}(\pi_S^H)$. Choosing $H$ amounts to choosing the number $j\in\{1,\ldots,\ell\}$ such that the northeast endpoint of $H$ is $(n+1+j,n+1+j)$. The permutation $\pi$ and the choice of $H$ determine the permutations $\pi_U^H$ and $\pi_S^H$. On the other hand, the choices of $H$ and the permutations $\pi_U^H$ and $\pi_S^H$ uniquely determine $\pi$. It follows that $B_\ell^{(g)}(n+1)$, which is the number of ways to choose an element of $s^{-1}(\mathcal D_\ell^{(g)}(n+1))$, is also the number of ways to choose the tuple $(j,\pi_U^H,\pi_S^H,\mu,\lambda)$. Let us fix a choice of $j$. 

Assume for the moment that $i\leq n-1$, and let $r$ be the largest entry appearing to the left of $n+1$ in $\pi$. Because $\pi$ is $2$-stack-sortable, we can use Theorem~\ref{Thm8} to see that $\pi_U^H$ is a permutation of the set $\{1,\ldots,n-i-1\}\cup\{r,n+1\}\cup\{n+2+j,\ldots,n+\ell+1\}$ and that $\pi_S^H$ is a permutation of $\{n-i,\ldots,n+j\}\setminus\{r,n+1\}$. Therefore, choosing $\pi_U^H$ and $\pi_S^H$ is equivalent to choosing their normalizations and the value of $r$. The normalization of $\pi_S^H$ is in $\mathcal D_{\geq j-1}^{(a)}(i)$ for some $a\in\{2,\ldots,i+1\}$, while the normalization of $\pi_U^H$ is in $\mathcal D_{\geq\ell-j+1}^{(b)}(n-i)$ for some $b\in\{2,\ldots,n-i+1\}$. Once we have chosen $a$ and $b$, the number of choices for the tuple $(\pi_U^H,\mu,\pi_S^H,\lambda)$ is $B_{\geq j-1}^{(a)}(i)B_{\geq \ell-j+1}^{(b)}(n-i)$.

Suppose we have already chosen the value of $a$. The fact that $\pi$ avoids $2341$ and the definition of a legal space tell us that there are $a$ possible values of $r$, say $\kappa_1<\cdots<\kappa_a$ (see Example~\ref{Exam1} for an illustration of this part of the proof). If we choose $r=\kappa_m$, then $\pi$ has $a+b-m+1+\ell$ legal spaces. We are assuming that $\leg(\pi)=\ell+g$, so $g=a+b-m+1$. It follows that $2\leq a\leq n$ and $\max\{2,g-a\}\leq b\leq g-1$. Since $a\in\{2,\ldots,i+1\}$ and $b\in\{2,\ldots,n-i+1\}$, we also have the constraint $a-1\leq i\leq n-b+1$. This explains the expression $\sum_{a=2}^{n}\sum_{b=\max\{2,g-a\}}^{g-1}\sum_{i=a-1}^{n-b+1}\sum_{j=1}^\ell B_{\geq j-1}^{(a)}(i) B_{\geq \ell-j+1}^{(b)}(n-i)$ in \eqref{Eq16}. 

The expression $\sum_{j=1}^\ell B_{\geq j-1}^{(g-1)}(n)C_{\ell-j+1}$ in \eqref{Eq16} comes from the case in which $i=n$. In this case, $\pi_S^H$ is in $\mathcal D_{\geq j-1}^{(g-1)}(n)$, and $\pi_U^H=(n+1)(n+2+j)(n+3+j)\cdots(n+\ell+1)$ is an increasing permutation of length $\ell-j+1$. The number of choices for the pair $(\pi_S^H,\lambda)$ is $B_{\geq j-1}^{(g-1)}(n)$. The number of choices for $\mu$ is $|s^{-1}(\pi_U^H)|=C_{\ell-j+1}$. 
\end{proof}

\begin{example}\label{Exam1}
Consider the part of the proof of Theorem~\ref{Thm6} in which we have already chosen $n,g,\ell,j,i$ and have assumed $i\leq n-1$. Suppose $n=8$, $\ell=5$, $j=2$, and $i=5$. If we choose the normalization of $\pi_U^H$ to be $24315678$ and choose the normalization of $\pi_S^H$ to be $315246$, then $a=\leg(315246)-(j-1)=5$ and $b=\leg(24315678)-(\ell-j+1)=4$. The green dots in Figure~\ref{Fig3} represent the possible choices for $r$, which are $\kappa_1=4$, $\kappa_2=5$, $\kappa_3=7$, $\kappa_4=8$, and $\kappa_5=9$. If $r=\kappa_m$, then we can refer to this figure to see that $\leg(\pi)=15-m=\ell+a+b-m+1$. Hence, the choice of $r$ is determined by the value of $g$.  

\begin{figure}[h]
\begin{center}
\includegraphics[width=.4\linewidth]{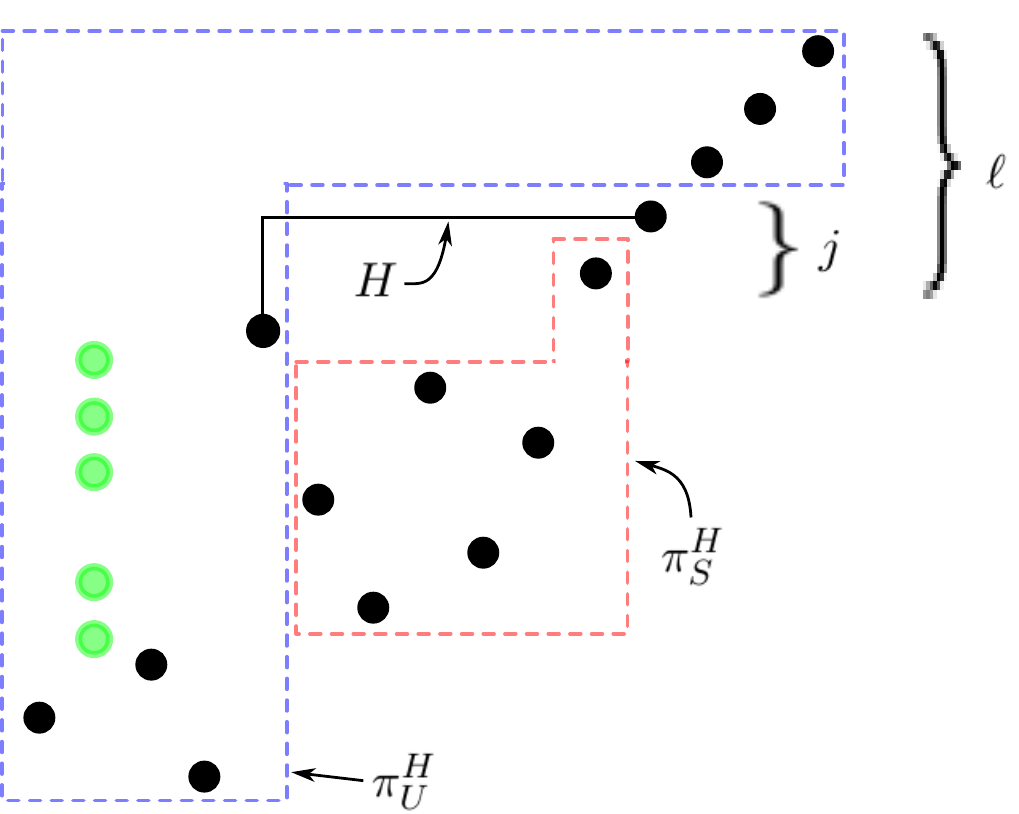}
\caption{The decomposition of $\pi$ into $\pi_U^H$ and $\pi_S^H$ along with the possible choices for $r$.}
\label{Fig3}
\end{center}  
\end{figure}
\end{example}

The proof of Theorem~\ref{Thm6} generalizes, allowing us to obtain a recurrence for $W_3(n,k,p)$, the number of $3$-stack-sortable permutations in $S_n$ with $k$ descents and $p$ peaks. We actually state the following theorem in terms of polynomials, but one can obtain the desired recurrence by comparing coefficients. In what follows, let \[E_{\geq\ell}^{(g)}(n)=\sum_{\sigma\in s^{-1}(\mathcal D_{\geq\ell}^{(g)}(n))}x^{\des(\sigma)+1}y^{\peak(\sigma)+1},\]
where $\mathcal D_{\geq\ell}^{(g)}(n)$ is as in \eqref{Eq18}. We have suppressed the dependence on $x$ and $y$ in our notation for readability. Let $L_r(x,y)$ be as in \eqref{Eq19}.    

\begin{theorem}\label{Thm9}
If $n\geq 1$, then \[\sum_{\sigma\in\mathcal W_3(n)}x^{\des(\sigma)+1}y^{\peak(\sigma)+1}=\sum_{g=1}^{n+1}E_{\geq 0}^{(g)}(n),\] where the polynomials $E_{\geq \ell}^{(g)}(n)$ satisfy the following relations. We have $E_{\geq\ell}^{(0)}(n)=0$ and \[E_{\geq\ell}^{(g)}(1)=\begin{cases} 0, & \mbox{if } g\neq 2; \\ L_{\ell+1}(x,y), & \mbox{if } g=2. \end{cases}\] If $n,g\geq 1$ and $\ell\geq 0$, then 
\[E_{\geq\ell}^{(g)}(n+1)=\sum_{j=1}^\ell\left(\sum_{a=2}^{n}\sum_{b=\max\{2,g-a\}}^{g-1}\sum_{i=a-1}^{n-b+1}E_{\geq j-1}^{(a)}(i) E_{\geq \ell-j+1}^{(b)}(n-i)+ E_{\geq j-1}^{(g-1)}(n)L_{\ell-j+1}(x,y)\right)\] \[+ E_{\geq\ell+1}^{(g-1)}(n).\] 
\end{theorem}

\begin{proof}
To derive the formula for $E_{\geq\ell}^{(g)}(1)$, we follow the same argument used to find the formula for $B_{\geq\ell}^{(g)}(1)$ in the proof of Theorem~\ref{Thm6}, except we use Remark~\ref{Rem1} instead of Theorem~\ref{Thm1}. To derive the last statement in this theorem, we follow the rest of the proof of Theorem~\ref{Thm6}, except we invoke the refined decomposition lemma instead of the decomposition lemma and again use Remark~\ref{Rem1} instead of Theorem~\ref{Thm1}.  
\end{proof}

\section{Data Analysis}\label{Sec:Data}

The \emph{sum} of two permutations $\mu$ and $\lambda$, denoted $\mu\oplus\lambda$, is the permutation whose plot is obtained by placing the plot of $\lambda$ above and to the right of the plot of $\mu$. It is easy to check that the sum of two $t$-stack-sortable permutations is $t$-stack-sortable. It follows that $W_t(m+n)\geq W_t(m)W_t(n)$ for all $m,n\geq 1$. We express this by saying the sequence $(W_t(n))_{n\geq 1}$ is \emph{supermultiplicative}. It follows from Fekete's lemma that 
\begin{equation}\label{Eq20}
\lim_{n\to\infty}\frac{W_t(n+1)}{W_t(n)}=\lim_{n\to\infty}W_t(n)^{1/n}=\sup_{n\geq 1}W_t(n)^{1/n}.
\end{equation}   

We have used Theorem~\ref{Thm6} to compute the numbers $W_3(n)$ for $n\leq 174$. We have added these terms to the OEIS sequence A134664. This allows us to prove the first nontrivial lower bound for $\lim\limits_{n\to\infty}W_3(n)^{1/n}$. Note that this is better than the lower bound of $(\sqrt{3}+1)^2$ obtained in Section~\ref{Sec:Lower}. 

\begin{theorem}\label{Thm11}
We have \[\lim_{n\to\infty}W_3(n)^{1/n}\geq 8.659702.\] 
\end{theorem}
\begin{proof}
The value of $W_3(174)$ is
\begin{equation*}
\begin{split}
&1335109055832443343636882328903941541
553316885478273864987091560565206631540
380152\\
 &787051400123018026588950184116831
251220601282385312955696662890107919486
8270269904,
\end{split}
\end{equation*}
and the $174^\text{th}$ root of this number is slightly more than $8.659702$. The proof follows from \eqref{Eq20}. 
\end{proof}

We can also show that B\'ona's Conjectures~\ref{Conj3} and \ref{Conj4} contradict each other. 

\begin{theorem}\label{Thm12}
If $(W_3(n))_{n\geq 1}$ is log-convex, then $W_3(n)>{4n\choose n}$ for all sufficiently large $n$.  
\end{theorem}

\begin{proof}
It follows from Stirling's formula that $\lim\limits_{n\to\infty}{4n\choose n}^{1/n}=256/27\approx 9.4815$. Also, $\dfrac{W_3(174)}{W_3(173)}\approx 9.4907$. If $(W_3(n))_{n\geq 1}$ is log-convex, then $\lim\limits_{n\to\infty}W_3(n)^{1/n}=\lim\limits_{n\to\infty}\dfrac{W_3(n+1)}{W_3(n)}\geq 9.4907>9.4815$.     
\end{proof}

We now turn our attention to the parity of $W_3(n)$ and B\'ona's Conjecture~\ref{Conj6}. Let $\varepsilon_t(n)$ be the number in $\{0,1\}$ with the same parity as $W_t(n)$. As mentioned in the introduction, $\varepsilon_t(n)=0$ whenever $n$ is even. The values of $\varepsilon_3(2n+1)$ for $0\leq n\leq 86$ are 
\begin{equation}\label{Eq21}
1, 0, 0, 0, 1, 0,
1, 1, 1, 1, 1, 1, 0, 0, 0, 1, 1, 0, 0, 1, 1,
1, 1, 1, 1, 1, 1, 0, 1, 1, 1, 1, 1, 0, 1, 1,
1, 1, 1, 0, 1, 0,
\end{equation} \[ 0, 1, 0, 1, 0, 1, 1, 1, 0,
1, 0, 0, 1, 0, 1, 1, 0, 1, 1, 0, 0, 0, 1, 1, 1, 1, 0, 1, 1, 0, 0, 1, 1, 0, 0, 1, 0, 0, 0, 0, 0, 1, 0, 1, 0.\] Letting $\mathfrak g_t(m)=\sum_{n=1}^m\varepsilon_t(n)$, we can use this data to find that $\mathfrak g_1(m)<\mathfrak g_2(m)\leq\mathfrak g_3(m)$ whenever $13\leq m\leq 660$. Therefore, it appears that $W_3(n)$ is odd more frequently than $W_2(n)$! If this is true, then B\'ona's Conjecture~\ref{Conj6} is certainly false. We state some new conjectures and open problems concerning the parities of the numbers $W_3(n)$ in Section~\ref{Sec:Conclusion}. 

Let us end this section by recording one final proposition, which verifies B\'ona's Conjecture~\ref{Conj5} in several new cases. We have obtained this proposition by computing several values of $W_3(n,k)$ via Theorem~\ref{Thm9} (setting $y=1$ in that theorem). 

\begin{proposition}\label{Prop1}
If $n\leq 43$, then the polynomial $\displaystyle\sum_{\sigma\in\mathcal W_3(n)}x^{\des(\sigma)}$ has only real roots.
\end{proposition} 

\section{Lower Bounds for $t$-Stack-Sortable Permutations}\label{Sec:Lower}

Let $\Gamma_t$ be the set of all $\kappa=\kappa_1\cdots\kappa_{t+2}\in S_{t+2}$ such that $\kappa_{t+1}=t+2$ and $\kappa_{t+2}=1$. Let $\Av_n(\Gamma_t)$ be the set of permutations in $S_n$ that avoid all of the patterns in $\Gamma_t$. After applying a dihedral symmetry to the permutations in $\Gamma_t$, we can use a result of Kremer \cite{Kremer1, Kremer2} to see that 
\begin{equation}\label{Eq22}
\sum_{n\geq t}|\Av_n(\Gamma_t)|x^n=(t-1)!x^{t-2}\frac{1+(t-1)x-\sqrt{1-2(t+1)x+(t-1)^2x^2}}{2}.
\end{equation}
Some basic singularity analysis now shows that $\lim\limits_{n\to\infty}|\Av_n(\Gamma_t)|^{1/n}=(\sqrt{t}+1)^2$. 

We will prove by induction that $\Av_n(\Gamma_t)\subseteq\mathcal W_t(n)$. Since $\Gamma_1=\{231\}$, this is certainly true for $t=1$ (by Theorem~\ref{Thm1}). Now suppose that $t\geq 2$ and that $\Av_n(\Gamma_{t-1})\subseteq\mathcal W_{t-1}(n)$. Choose a permutation $\pi\in S_n\setminus\mathcal W_t(n)$. This means that $s(\pi)\not\in\mathcal W_{t-1}(n)$, so $s(\pi)$ contains a permutation in $\Gamma_{t-1}$. In other words, there exist entries $b_1,\ldots,b_{t-1},c,a$ that appear in this order in $s(\pi)$ and satisfy $a<b_j<c$ for all $j\in\{1,\ldots,t-1\}$. Because $c$ appears to the left of $a$ in $s(\pi)$, there must be an entry $d>c$ that appears to the right of $c$ and to the left of $a$ in $\pi$. The entries $b_1,\ldots,b_{t-1}$ must appear to the left of $d$ in $\pi$ since they would appear to the right of $c$ in $s(\pi)$ otherwise. The subpermutation of $\pi$ formed by the entries $a,b_1,\ldots,b_{t-1},c,d$ has a normalization that is in $\Gamma_t$, so $\pi\not\in\Av_n(\Gamma_t)$. This completes the induction and proves the following theorem. 

\begin{theorem}\label{Thm4}
For every $t\geq 1$, we have \[\lim_{n\to\infty}W_t(n)^{1/n}\geq(\sqrt t+1)^2.\]
\end{theorem}

In \cite{Smith}, Smith investigated a variant of the stack-sorting map known as the ``left-greedy algorithm." Let $\widehat{\mathcal W}_t(n)$ be the set of permutations in $S_n$ that can be sorted by $t$ stacks in series using the left-greedy algorithm (see her paper for definitions). Smith proved that $\mathcal W_t(n)\subseteq\widehat{\mathcal W}_t(n)$ and that $|\widehat{\mathcal W}_t(n)|\geq \dfrac{t!}{(t+1)^t}(t+1)^n$ whenever $n\geq t\geq 1$. In terms of exponential growth rates, this shows that $\lim\limits_{n\to\infty}|\widehat{\mathcal W}_t(n)|^{1/n}\geq t+1$ (using Fekete's lemma, one can show that this limit exists). The following corollary of Theorem~\ref{Thm4} improves this estimate. 

\begin{corollary}\label{Cor2}
For every $t\geq 1$, we have $\lim\limits_{n\to\infty}|\widehat{\mathcal W}_t(n)|^{1/n}\geq (\sqrt t+1)^2$. 
\end{corollary}

\section{Symmetry, Unimodality, $\gamma$-Nonnegativity, Log-Concavity, and Real-Rootedness}\label{Sec:Symmetry}

We devote this brief section to showing how Br\"and\'en's theorem concerning $\gamma$-nonnegativity (Theorem~\ref{Thm3}) follows easily from Theorem~\ref{Thm10}. We also show that the analogue of that theorem with ``$\gamma$-nonnegative" replaced by ``real-rooted" is false. Let us begin by recalling some definitions. 

A polynomial $p(x)=a_0+a_1x+\cdots+a_nx^n\in\mathbb R_{\geq 0}[x]$ is called \begin{itemize}
\item \emph{symmetric} if $a_i=a_{n-i}$ for all $i\in\{0,\ldots,n\}$; in this case, $n/2$ is called the \emph{center of symmetry} of $p(x)$;
\item \emph{unimodal} if there exists $j\in\{0,\ldots,n\}$ such that $a_0\leq a_1\leq\cdots\leq a_j\geq a_{j+1}\geq\cdots\geq a_n$;
\item \emph{log-concave} if $a_{i-1}a_{i+1}\leq a_i^2$ for all $i\in\{1,\ldots,n-1\}$;
\item \emph{real-rooted} if all of the complex roots of $p(x)$ are real.
\end{itemize}
If $p(x)$ is a symmetric polynomial with center of symmetry $n/2$, then it can be written in the form $p(x)=\sum_{m=0}^{\left\lfloor n/2\right\rfloor}\gamma_m x^m(1+x)^{n-2m}$ for some real numbers $\gamma_m$. We then say $p(x)$ is $\gamma$-\emph{nonnegative} if the numbers $\gamma_m$ are all nonnegative. We have the following implications among these properties for polynomials in $\mathbb R_{\geq 0}[x]$ \cite{Branden2}: \[\text{real-rooted}\Longrightarrow\text{log-concave}\Longrightarrow\text{unimodal};\] \[\text{symmetric and real-rooted}\Longrightarrow \gamma\text{-nonnegative}\Longrightarrow\text{symmetric and unimodal}.\]

\begin{proof}[New Proof of Theorem~\ref{Thm3}]\hspace{-.2cm}\footnote{To deduce B\'ona's symmetry and unimodality result from Theorem~\ref{Thm10}, one simply needs to observe that this theorem tells us that $\sum_{\sigma\in \mathcal W_t(n)}x^{\des(\sigma)+1}$ is a sum of products of Narayana polynomials with the same center of symmetry and then use the well-known fact that Narayana polynomials are real-rooted.}
Note that it suffices to prove Theorem~\ref{Thm3} in the specific case in which $A=\{\pi\}$ is a singleton set. Indeed, the result for a general set $A\subseteq S_n$ then follows by summing over all $\pi\in A$. Thus, let us fix a permutation $\pi\in S_n$ with exactly $k$ descents. 

Recall the notation from \eqref{Eq33} and \eqref{Eq34}.
One can show that \[N_q(x)=\sum_{m=0}^q\frac{V(q,m+1)}{2^{q-1-2m}}x^{m+1}(1+x)^{q-1-2m}\] for all $q\geq 1$. Therefore, for $(q_0,\ldots,q_k)\in\Comp_{k+1}(n-k)$, we have \[\prod_{t=0}^kN_{q_t}(x)=\prod_{t=0}^k\sum_{m_t=0}^{q_t}\frac{V(q_t,m_t+1)}{2^{q_t-1-2m_t}}x^{m_t+1}(1+x)^{q_t-1-2m_t}\] \[=\sum_{m=0}^n\:\sum_{\substack{m_0+\cdots+m_k=m-k \\ m_0,\ldots,m_k\geq 0}}\frac{1}{2^{n-1-2m}}\left(\prod_{t=0}^kV(q_t,m_t+1)\right)x^{m+1}(1+x)^{n-1-2m}.\] Let $\mathcal V(\pi)\subseteq\Comp_{k+1}(n-k)$ be the set of compositions from Theorem~\ref{Thm10}. Invoking equation \eqref{Eq12} from that theorem, we obtain \[\sum_{\sigma\in s^{-1}(\pi)}x^{\des(\sigma)+1}=\sum_{(q_0,\ldots,q_k)\in\mathcal V(\pi)}\prod_{t=0}^kN_{q_t}(x)\] \[=\sum_{(q_0,\ldots,q_k)\in\mathcal V(\pi)}\sum_{m=0}^n\:\sum_{\substack{m_0+\cdots+m_k=m-k \\ m_0,\ldots,m_k\geq 0}}\frac{1}{2^{n-1-2m}}\left(\prod_{t=0}^kV(q_t,m_t+1)\right)x^{m+1}(1+x)^{n-1-2m}\] \[=\sum_{m=0}^n\frac{1}{2^{n-1-2m}}x^{m+1}(1+x)^{n-1-2m}\sum_{(q_0,\ldots,q_k)\in\mathcal V(\pi)}\sum_{(m_0',\ldots,m_k')\in\Comp_{k+1}(m+1)}\prod_{t=0}^kV(q_t,m_t'),\] where we have made the substitution $m_i'=m_i+1$. It turns out that \[\sum_{(q_0,\ldots,q_k)\in\mathcal V(\pi)}\sum_{(m_0',\ldots,m_k')\in\Comp_{k+1}(m+1)}\prod_{t=0}^kV(q_t,m_t')\] is the coefficient of $y^{m+1}$ in the polynomial on the right-hand side of \eqref{Eq13}, so it is equal to \[|\{\sigma\in s^{-1}(\pi):\peak(\sigma)=m\}|.\] Note that this is $0$ if $m>\frac{n-1}{2}$. Hence, \[\sum_{\sigma\in s^{-1}(\pi)}x^{\des(\sigma)}=\sum_{m=0}^{\left\lfloor\frac{n-1}{2}\right\rfloor}\frac{|\{\sigma\in s^{-1}(\pi):\peak(\sigma)=m\}|}{2^{n-1-2m}}x^m(1+x)^{n-1-2m}. \qedhere\]
\end{proof} 

We now give an example to show that Theorem~\ref{Thm3} is false if the term ``$\gamma$-nonnegative" is replaced by ``real-rooted." 

\begin{example}
Let \[\mu=6\,\,7\,\,8\,\,4\,\,5\,\,9\,\,10\,\,1\,\,2\,\,3\,\,11\quad\text{and}\quad\mu'=6\,\,7\,\,8\,\,9\,\,2\,\,3\,\,4\,\,5\,\,10\,\,1\,\,11.\] We claim that $\sum_{\sigma\in s^{-1}(\{\mu,\mu'\})}x^{\des(\sigma)}$ is not real rooted. To see this, we use the fact\footnote{The reader interested in seeing why this is the case can refer to \cite{DefantClass} for the full definition of $\mathcal V(\pi)$ and a description of how to compute it. However, the reader wishing to avoid this definition can still compute $\sum_{\sigma\in s^{-1}(\{\mu,\mu'\})}x^{\des(\sigma)}$ using a brute-force computer program that simply finds all of the permutations in $s^{-1}(\{\mu,\mu'\})$. \emph{A priori}, a brute-force computer program would not easily \emph{find} this example since it would have to search over \emph{subsets} of $S_{11}$.} that $\mathcal V(\mu)=\{(4,2,3),(3,3,3)\}$ and $\mathcal V(\mu')=\{(4,4,1)\}$. Using \eqref{Eq12}, we find that \[\sum_{\sigma\in s^{-1}\{\mu,\mu'\}}x^{\des(\sigma)}=\frac{1}{x}\sum_{\sigma\in s^{-1}(\mu)}x^{\des(\sigma)+1}+\frac{1}{x}\sum_{\sigma\in s^{-1}(\mu')}x^{\des(\sigma)+1}\] \[=\frac{1}{x}\left(N_4(x)N_2(x)N_3(x)+N_3(x)N_3(x)N_3(x)\right)+\frac{1}{x}N_4(x)N_4(x)N_1(x)\] \[=3 x^2 + 31 x^3 + 112 x^4 + 169 x^5 + 112 x^6 + 31 x^7 + 3 x^8,\] and this polynomial is not real-rooted. This example yields a negative answer to the last part of Question 12.1 in \cite{DefantClass}. 
\end{example}

\begin{remark}\label{Rem2} 
Theorem~\ref{Thm3} diverges from B\'ona's point of view in Conjecture~\ref{Conj5} by replacing the sum over $s^{-1}(\mathcal W_{t-1}(n))$ with a sum over $s^{-1}(A)$ for an arbitrary set $A\subseteq S_n$. This different viewpoint suggests that the sets of the form $\mathcal W_t(n)=s^{-1}(\mathcal W_{t-1}(n))$ might not be too special when compared with arbitrary sets of the form $s^{-1}(A)$ for $A\subseteq S_n$. 
If one believes Conjecture~\ref{Conj5}, then the preceding example lends credence to the hypothesis that the sets $\mathcal W_t(n)$ are special. On the other hand, if one does not believe there is anything special about the sets $\mathcal W_t(n)$, then this example hints that Conjecture~\ref{Conj5} might be false. 
\end{remark}

\section{Conjectures and Open Problems}\label{Sec:Conclusion}

We saw in Theorem~\ref{Thm12} that Conjectures~\ref{Conj3} and \ref{Conj4} cannot both be true. Our data suggests that Conjecture~\ref{Conj4} is true. Moreover, by plotting the points $(1/n,W_3(n))$ for $1\leq n\leq 174$, we have arrived at the following new conjecture.  

\begin{conjecture}\label{Conj11}
We have \[9.702<\lim_{n\to\infty}W_3(n)^{1/n}<9.704.\]
\end{conjecture}

We also believe that the decomposition lemma could be used (possibly along with a significant amount of work) to find a lower bound for $\lim\limits_{n\to\infty}W_4(n)^{1/n}$ that exceeds $9.704$. 

Turning back to the parities of the numbers $W_3(n)$, we have the following problem.  
\begin{problem}\label{Prob1}
Characterize the positive integers $n$ such that $W_3(n)$ is odd. 
\end{problem}

Problem~\ref{Prob1} seems more tractable now that we have obtained a recurrence for the numbers $W_3(n)$ in Theorem~\ref{Thm6}. Indeed, it appears as though there could be some patterns in the sequence whose initial terms are listed in \eqref{Eq21}. Solving this problem could require going through the proof of Theorem~\ref{Thm6} and seeing which terms in the various sums simplify when we reduce modulo $2$. 

Recall the definition of $\mathfrak g_t(m)$ from Section~\ref{Sec:Data}. We have the following conjectures. Conjectures~\ref{Conj8}, \ref{Conj9}, and \ref{Conj10} each contradict B\'ona's Conjecture~\ref{Conj6}. 

\begin{conjecture}\label{Conj7}
The limit $\displaystyle\lim_{n\to\infty}\frac{\log \mathfrak g_3(m)}{\log m}$ exists. 
\end{conjecture}

\begin{conjecture}\label{Conj8}
We have $\displaystyle\liminf_{n\to\infty}\frac{\log \mathfrak g_3(m)}{\log m}>0$.
\end{conjecture}

\begin{conjecture}\label{Conj9}
For every integer $m\geq 13$, we have $\mathfrak g_2(m)\leq\mathfrak g_3(m)$. 
\end{conjecture}

\begin{conjecture}\label{Conj10}
We have $\lim\limits_{m\to\infty}(\mathfrak g_3(m)-\mathfrak g_2(m))=\infty$. 
\end{conjecture}

\section{Acknowledgments}
I would like to express my deepest gratitude to Niven Achenjang, Amanda Burcroff, and Eric Winsor for writing computer programs that calculated the numbers $W_3(n)$ much faster than the author's original program. I would also like to thank Jay Pantone for running one of these programs on his computer for several days and for analyzing the resulting data. The contributions that these people made were paramount to the analysis discussed in Section~\ref{Sec:Data}. I thank Mikl\'os B\'ona and Doron Zeilberger for helpful conversations. I also thank Caleb Ji, who wrote a poem that inexplicably predicted I would make progress in the study of $3$-stack-sortable permutations. I thank the anonymous referees for helpful comments that improved the presentation of this article. 

The author was supported by a Fannie and John Hertz Foundation Fellowship and an NSF Graduate Research Fellowship.

\end{document}